\documentclass[10pt, reqno]{amsart}
\usepackage{amsmath}
\usepackage{amssymb}
\usepackage{enumerate}
\usepackage{amsbsy}
\usepackage{amsfonts}
\usepackage{color}
\usepackage{dsfont,amsxtra,latexsym,amscd,amsthm}

\newtheorem{thm}{Theorem}[section]
\newtheorem{prop}[thm]{Proposition}
\newtheorem{lem}[thm]{Lemma}

\newtheorem{cor}[thm]{Corollary}

\numberwithin{equation}{section}

\theoremstyle{remark}
\newtheorem{rem}{Remark}

\newcommand{\supp}{\text{supp }}

\newcommand{\R}{\mathbb{R}}
\newcommand{\C}{\mathbb{C}}
\newcommand{\Z}{\mathbb{Z}}

\newcommand{\ls}{\lesssim}

\begin{document}
\title[Strong Convergence for Discrete NLS in the Continuum Limit]
{Strong Convergence for Discrete Nonlinear Schr\"odinger equations in the Continuum Limit}

\author{Younghun Hong}
\address{Department of Mathematics, Chung-Ang University, Seoul 06974, Republic of Korea}
\email{yhhong@cau.ac.kr}

\author{Changhun Yang}
\address{Department of Mathematical Sciences, Seoul National University, Seoul 151-747, Republic of Korea}
\email{maticionych@snu.ac.kr}

\subjclass[2010]{}

\begin{abstract}
We consider discrete nonlinear Schr\"odinger equations (DNLS) on the lattice $h\mathbb{Z}^d$ whose linear part is determined by the discrete Laplacian which accounts only for nearest neighbor interactions, or by its fractional power. We show that in the continuum limit $h\to 0$, solutions to DNLS converge strongly in $L^2$ to those to the corresponding continuum equations, but a precise rate of convergence is also calculated. In particular cases, this result improves weak convergence in Kirkpatrick, Lenzmann and Staffilani \cite{KLS}. Our proof is based on a suitable adjustment of dispersive PDE techniques to a discrete setting. Notably, we employ uniform-in-$h$ Strichartz estimates for discrete linear Schr\"odinger equations in \cite{HY}, which quantitatively  measure dispersive phenomena on the lattice. Our approach could be adapted to a more general setting like \cite{KLS} as long as the desired Strichartz estimates are obtained.
\end{abstract}

\maketitle

\section{Introduction}

In applications, a discrete equation is often introduced as a simplified model for a given physical equation. Indeed, spatial discretization would be a first step to implement finite difference methods (FDM), transferring an equation on a continuum domain to that on a lattice domain. If the domain is unbounded, taking Dirichlet cut-off for finitization, the equation becomes suitable for numerical simulation by the method of lines (MOL), or it could be simplified further by time discretization. Therefore, both in theory and practice, important is a rigorous proof of convergence from solutions to a discrete equation to those to a continuum equation as the distance between lattice points (or the size of grid) gets smaller and smaller. This convergence is referred to as a \textit{continuum limit}.

In this paper, we consider a class of nonlinear dispersive equations, in particular, a nonlinear Schr\"odinger equation (NLS) with power-type nonlinearity, 
\begin{equation}\label{NLS}
i\partial_t u=(-\Delta)^\alpha u+\lambda |u|^{p-1}u
\end{equation}
on the Euclidean domain $\mathbb{R}^d$, where $0<\alpha\leq 1$, $\alpha\neq \frac{1}{2}$, $p>1$, $\lambda\in\R/\{0\}$ and
$$u=u(t,x):\mathbb{R}\times \mathbb{R}^d\to\mathbb{C}.$$
Here, the fractional Laplacian $(-\Delta)^\alpha$ is the Fourier multiplier operator with symbol $|\xi|^{2\alpha}$. The standard NLS ($\alpha=1$) is derived as a mean-field equation for Bose-Einstein condensates, and it also appears in nonlinear optics to describe wave propagation in a weakly nonlinear medium \cite{SS}. The fractional NLS $(\frac{1}{2}<\alpha<1)$ was introduced by Laskin to describe fractional quantum mechanics \cite{L}. A model for dispersive wave turbulence also has a fractional dispersion relation \cite{MMT}. The case $d=1$ and $\alpha=\frac14$ is considered as a simplifed model equation for the two-dimensional water wave equation \cite{IP}.

As a discretization of the equation \eqref{NLS}, we consider a discrete nonlinear Schr\"odinger equation (DNLS)
\begin{equation}\label{DNLS}
i\partial_t u_h=(-\Delta_h)^\alpha u_h+\lambda |u_h|^{p-1}u_h
\end{equation}
on the lattice $h\mathbb{Z}^d$, where
$$u_h=u_h(t,x):\mathbb{R}\times h\mathbb{Z}^d\to\mathbb{C}.$$
Indeed, there are several natural ways to define a Laplacian operator on a lattice, but we here restrict ourselves to the simplest but perhaps the most important one given by 
\begin{equation}\label{discrete Laplacian}
(\Delta_h f)(x)=\sum_{j=1}^d\frac{f(x+he_j)+f(x-he_j)-2f(x)}{h^2},\quad\forall x\in h\mathbb{Z}^d,
\end{equation}
which accounts only for nearest neighbor interactions. A nonlocal fractional Laplacian is then properly defined by means of functional calculus. The discrete model \eqref{DNLS} formally converges to the continuum model \eqref{NLS} as $h\to 0$. It should be noted that not only for numerical experiments, DNLS is also physically important by itself for optical lattices and for charge transport in biopolymers like the DNA \cite{GCR1, GCR2, MSGJR}. There is a huge physics literature on this topic, and we refer to \cite{EJ, KRB, Kev} for overview. Thus, conversely, NLS \eqref{NLS} can be introduced to describe the limiting dynamics of a physical discrete model \cite{KLS}.

The goal of this paper is to develop a general strategy to give a rigorous proof of the continuum limits of discrete nonlinear dispersive equations. Indeed, continuum limits for ground state solitons \cite{FP, JW1, JW2} and those for solutions near soliton manifolds \cite{BFG} are now relatively well-understood in various contexts. Nevertheless, as for continuum limits of general solutions, to the best of the authors' knowledge, the only known result is due to Kirkpatrick, Lenzmann and Staffilani \cite{KLS}. In this important work, it is proved that solutions to a one-dimensional cubic DNLS, including a very large class of long-range interactions, weakly converge to solutions to the corresponding fractional NLS as $h\to 0$.

Our main result asserts that restricting to the particular choice of the Laplacian \eqref{discrete Laplacian} and its fractional power, weak convergence in the previous work \cite{KLS} can be improved to strong convergence. Furthermore, a precise rate of convergence is calculated. Our approach is based on a suitable adjustment of dispersive PDE techniques to problems on lattices in consideration of their limits, which involves ``uniform-in-$h$" Strichartz estimates for discrete linear Schr\"odinger equations (see Theorem \ref{Strichartz} below). 

For the statement, the following definitions are needed to relate functions on a lattice to those on the whole space. Given $f\in L^2(\mathbb{R}^d; \mathbb{C})$, we define its \textit{discretization} $f_h: h\mathbb{Z}^d\to\mathbb{C}$ by 
\begin{equation}\label{discretization}
f_h(x_m):=\frac{1}{h^d}\int_{x_m+[0,h)^d} f(x)\ dx,\quad \forall x_m=hm\in h\mathbb{Z}^d.
\end{equation}
Conversely, we define the \textit{linear interpolation operator} $p_h$ sending a function $f:h\mathbb{Z}^d\to\mathbb{C}$ on the lattice to a function on $\mathbb{R}^d$, 
\begin{equation}\label{p_h}
(p_hf)(x):=f(x_m)+\sum_{j=1}^d\frac{f(x_m+he_j)-f(x_m)}{h} (x-x_m)_j,\quad\forall x\in x_m+[0,h)^d,
\end{equation}
where $x_j$ denotes the $j$-th component of $x\in\mathbb{R}^d$. The main theorem of this paper then reads as follows.

\begin{thm}[Continuum limits]\label{main theorem}
In the NLS case $(\alpha=1)$, we assume that $d=1,2,3$ and
	\begin{equation}\label{assumption 1}
	\left\{
	\begin{aligned}
	\max\left\{\frac{d-2}{d+2}, 0\right\}&<\frac{1}{p}<1&& \textup{when }\lambda>0&& \textup{(defocusing)},\\
	\frac{d}{d+4}&<\frac{1}{p}<1&&\textup{when }\lambda<0&& \textup{(focusing)}.
	\end{aligned}
	\right.
	\end{equation}
In the fractional NLS case $(0<\alpha<1)$, we assume that $d=1$, $\frac{1}{3}<\alpha<1$ and $\alpha\neq\frac{1}{2}$, and
	\begin{equation}\label{assumption 2}
	\left\{
	\begin{aligned}
	\max\left\{\frac{1-2\alpha}{1+2\alpha},0\right\} &<\frac{1}{p}<1&& \textup{when }\lambda>0&& \textup{(defocusing)},\\
	\frac{1}{1+4\alpha}&<\frac{1}{p}<1&&\textup{when }\lambda<0&& \textup{(focusing)}.
	\end{aligned}
	\right.
	\end{equation}
Let $h\in(0,1]$. Given initial data $u_0\in H^\alpha(\mathbb{R}^d)$, let $u(t)\in C(\mathbb{R}; H^\alpha(\mathbb{R}^d))$ be the global solution to NLS \eqref{NLS} (see Proposition \ref{WP:NLS}), and let $u_h(t)$ be the global solution to DNLS \eqref{DNLS} whose initial data $u_{h,0}$ is the discretization of $u_0$ (see Proposition \ref{GWP}). Then, there exist constants $A,B>0$, independent of $h$, such that for all $t\in\mathbb{R}$,
$$\|p_hu_{h}(t)-u(t)\|_{L^2(\mathbb{R}^d)}\leq A h^{\frac{\alpha}{1+\alpha}}e^{B|t|}\left(1+\|u_{0}\|_{H^\alpha (\mathbb{R}^d)}\right)^p.$$
\end{thm}

\begin{rem}
\begin{enumerate}[(i)]
\item As for NLS $(\alpha=1)$, the assumptions \eqref{assumption 1} are almost optimal in one to three dimensions in the sense that the full range for global well-posedness of the continuum equation \eqref{NLS} is covered except the defocusing energy-critical nonlinearity, i.e., $\lambda>0$, $d=3$ and $p=5$. Higher dimensions $d\geq 4$ are excluded due to a technical reason (see Remark \ref{high d restriction}). 
\item As for the fractional NLS $(0<\alpha<1)$, the assumptions \eqref{assumption 2} are also almost optimal in one dimension in that such conditions are currently required for global well-posedness of the continuum fractional NLS. Multi-dimensions $d\geq 2$ are not included here because of lack of uniform Strichartz estimates at this moment (see Remark \ref{Strichartz estimates remarks} (iii) below).
\end{enumerate}
\end{rem}

As mentioned above, the key new analysis tool of this paper is the following Strichartz estimates for discrete linear Schr\"odinger equations, which hold uniformly in $h>0$. We denote by $e^{-it(-\Delta_h)^\alpha}f$ the solution to the discrete linear Schr\"odinger equation $i\partial_tu_h=(-\Delta_h)^\alpha u_h$ with initial data $f$. We say that $(q,r)$ is \textit{admissible} if $2\leq q,r\leq\infty$,
\begin{equation}\label{admissible}
\frac{2}{q}+\frac{d}{r}=\frac{d}{2}\textup{ and }(q,r,d)\neq (2,\infty,2),
\end{equation}
and that $(q,r)$ is \textit{resonance admissible} if $2\leq q,r\leq\infty$,
\begin{equation}\label{r-admissible}
\frac{3}{q}+\frac{d}{r}=\frac{d}{2}\textup{ and }(q,r,d)\neq (2,\infty,3).
\end{equation}
We define the Lebesgue space $L_h^p$ on the lattice as the Banach space equipped with the norm
\begin{equation}\label{L_h^p}
\|f\|_{L_h^p}
:=\left\{\begin{aligned}
&\bigg\{h^d\sum_{x_m\in h\mathbb{Z}^d}|f(x_m)|^p\bigg\}^{1/p}&&\textup{if }1\leq p<\infty,\\
&\sup_{x_m\in h\mathbb{Z}^d}|f(x_m)|&&\textup{if }p=\infty,
\end{aligned}\right.
\end{equation}
and define the fractional derivative $|\nabla_h|^s$ as the Fourier multiplier of symbol $|\xi|^s$ via the discrete Fourier transform (see Section 2).

\begin{thm}[Uniform Strichartz estimates on a lattice]\label{Strichartz} Suppose that $1\leq d\leq 3$, $h>0$ and $0<\alpha\leq 1$ with $\alpha\neq\frac{1}{2}$. Then, there exists $C>0$, independent of $h>0$, such that the following hold.
\begin{enumerate}[(i)]
\item If $d=1,2,3$ and $\alpha=1$ or if $d=1$ and $\frac{1}{2}<\alpha<1$, then for any resonance admissible pair $(q,r)$, we have
\begin{equation}\label{Strichartz estimate 1}
\| e^{-it(-\Delta_h)^\alpha}f\|_{L_t^q(\mathbb{R};L_h^r)}\leq C \||\nabla_h|^{\frac{3-2\alpha}{q}}f\|_{L_h^2}.
\end{equation}
\item If $d=1$ and $0<\alpha<\frac{1}{2}$, then for any admissible pair $(q,r)$, we have
\begin{equation}\label{Strichartz estimate 2}
\| e^{-it(-\Delta_h)^\alpha}f\|_{L_t^q(\mathbb{R};L_h^r)}\leq C \| |\nabla_h|^{\frac{2(1-\alpha)}{q}}f\|_{L_h^2}.
\end{equation}
\end{enumerate}
\end{thm}

\begin{rem}\label{Strichartz estimates remarks}
\begin{enumerate}[(i)]
\item For fixed $h>0$ and $\alpha=1$, Strichartz estimates on the lattice $h\mathbb{Z}^d$ are established in Stefanov-Kevrekidis \cite{SK}. However, their constants blow up as $h\to 0$, so they cannot be directly applied to continuum limit problems. In our previous work \cite{HY}, developing harmonic analysis tools on the lattice $h\mathbb{Z}^d$, it is first observed that such inequalities may hold uniformly in $h>0$ paying additional fractional derivatives on the right hand side. Extending this result, in this paper, we obtain uniform Strichartz estimates for the one-dimensional discrete fractional Schr\"odinger equation, i.e., the case $d=1$, $0<\alpha<1$ and $\alpha\neq\frac{1}{2}$.
\item The admissible conditions \eqref{r-admissible} are different from those for the continuum equation \eqref{admissible}.
It is because the phase function in the integral representation of the solution $e^{-it(-\Delta_h)^\alpha}f$ via the discrete Fourier transform may have degenerate Hessian, and thus it only enjoys weaker dispersion. Such a phenomenon is sometimes referred to as \textit{lattice resonances}. Therefore, to compensate weaker dispersion, additional fractional derivatives are required on the right hand side for uniformity of Strichartz estimates. Interesting is absence of lattice resonances when interactions are more nonlocal, i.e., $0<\alpha<\frac{1}{2}$ (see the proof of Theorem \ref{Strichartz} $(ii)$).
\item Due to technical difficulties, fractional Schr\"odinger equations in multi-dimensions, that is, $0<\alpha<1$ and $d\geq 2$, are not included in Theorem \ref{Strichartz}. Indeed, the oscillatory integral associated with the fundamental solution of \eqref{DNLS} also may have degenerate Hessian. However, the phase function having degenerate Hessian is much more complicated to deal with in multi-dimensions (see \cite{BG} for instance).
\end{enumerate}
\end{rem}

The argument to justify the continuum limit (Theorem \ref{main theorem}) can be summarized as follows. First, as a direct consequence of the Strichartz estimates in Theorem \ref{Strichartz}, we get a ``time-averaged'' uniform-in-$h$ $L_h^\infty$-bound on solutions to discrete linear Schr\"odinger equations (Corollary \ref{linear L^infty bound}). Then, appying it to the nonlinear problem \eqref{DNLS}, we show that nonlinear solutions also satisfy a similar uniform $L_h^\infty$-bound (Proposition \ref{L^infty bound}). Having a better uniform bound at hand, we directly estimate the difference between two solutions in integral forms,
$$u(t)=e^{-it(-\Delta)^\alpha}u_0-i\lambda\int_0^t e^{-i(t-s)(-\Delta)^\alpha}(|u|^{p-1}u)(s)ds$$
and
$$p_hu_h(t)=p_he^{-it(-\Delta_h)^\alpha}u_{h,0}-i\lambda p_h\left[\int_0^t e^{-i(t-s)(-\Delta_h)^\alpha}(|u_h|^{p-1}u_h)(s)ds\right],$$
with some error estimates concerning the linear interpolation operator $p_h$ (see Proposition \ref{linear approx} and \ref{ph distribution}). Then, Theorem \ref{main theorem} follows by the standard Gronwall's lemma.

\begin{rem}

\begin{enumerate}[(i)]
\item The argument in \cite{KLS} relies on the uniform Sobolev inequality $\|f\|_{L_h^\infty}\lesssim \|f\|_{H_h^\alpha}$ (see Proposition \ref{GN and Sobolev}) and a uniform $H_h^\alpha$-bound on solutions to DNLS, which comes from the conservation laws. However, if $\alpha\leq\frac{d}{2}$, then such a Sobolev inequality fails, since so does $\|f\|_{L^\infty(\mathbb{R}^d)}\lesssim \|f\|_{H^\alpha(\mathbb{R}^d)}$. Hence, the assumptions $d=1$ and $\frac{1}{2}<\alpha\leq 1$ had to be imposed. Nevertheless, it turns out that thanks to dispersion, solutions to DNLS can be bounded uniformly in $L_h^\infty$ in a time-average sense even when $\alpha$ is smaller or in multi-dimensions. This observation not only allows us to extend the range of $d$ and $\alpha$, but also improves convergence in the continuum limit.
\item We believe that our stategy is robust, and it can be applied to other continuum limit problems. Indeed, the restriction to the choice of the Schr\"odinger operator $(-\Delta_h)^\alpha$ in this paper comes only from that uniform Strichartz estimates are currently available only in this case. Therefore, rigorous derivation of the fractional NLS (with a precise rate of convergence) in a more general setup as in \cite{KLS} could be reduced to proving uniform Strichartz estimates for the corresponding discrete linear flow.
\end{enumerate}
\end{rem}

\subsection{Organization of the paper}
The organization of this paper is as follows. In Section 2, we recall basic analysis tools on a lattice from \cite{Chat, KLS, HY}. In Section 3, we prove uniform Strichartz estimates (Theorem \ref{Strichartz}), which is the key inequality in this paper. Then, in Section 4, we show that nonlinear solutions also satisfy a time-averaged uniform $L_h^\infty$-bound. In Section 5, we discuss some important properties of discretization and linear interpolation. Finally, in Section 6, collecting all, we prove the main theorem (Theorem \ref{main theorem}).

\subsection{Notations}
We denote $A\lesssim B$ if there is a constant $C>0$, independent of $h>0$, such that $A\leq CB$, and denote $A\sim B$ if $A\lesssim B$ and $B\lesssim A$. With abuse of notation, we denote by $u_h(t)$ the solution to DNLS \eqref{DNLS} with initial data $u_{h,0}$, which is the discretization of $u_0$. Thus, $u_h(t)$ does not mean by the discretization of the solution $u(t)$ to NLS \eqref{NLS}.

\subsection{Acknowledgement}
This research of the first author was supported by Basic Science Research Program through the National Research Foundation of Korea(NRF) funded by the Ministry of Education (NRF-2017R1C1B1008215). The second author was supported in part by Samsung Science and Technology Foundation under Project Number SSTF-BA1702-02.

\section{Preliminaries}\label{Preliminaries}

In this section, we briefly review basic analysis tools on the lattice domain from \cite{Chat, KLS, HY}.

\subsection{Basic theory}
Let $h>0$ and $p\geq 1$. On the lattice $h\mathbb{Z}^d$, the natural Lebesgue space $L_h^p$ is defined by the collection of complex-valued functions on $h\mathbb{Z}^d$ equipped with the norm \eqref{L_h^p}. This function space is more or less the $\ell^p$-space of sequences having $d$ indices, because $\|f\|_{L_h^p}=h^{d/p}\|f\|_{\ell_x^p}$. Indeed, the Riemann sum for $|f(x)|^p$ on $\mathbb{R}^d$ is given by $h^{d}\sum_{x\in h\mathbb{Z}^d}|f(x)|^p$. Thus, putting $h^{d/p}$ in the norm is natural in the context of the continuum limit $h\to0$. It is easy to see from its connection to the $\ell^p$-space that we have H\"older's inequality
$$\|fg\|_{L_h^p} \le \|f\|_{L_h^{p_1}}\|g\|_{L_h^{p_2}},\quad \tfrac{1}{p}=\tfrac{1}{p_1}+\tfrac{1}{p_2},$$
and the standard duality relation
$$\|f\|_{L_h^p} =\sup_{\|g\|_{L_h^{p'}}\le 1} h^d\sum_{x\in h\mathbb{Z}^d}f(x)\overline{g(x)},\quad \tfrac{1}{p}+\tfrac{1}{p'}=1.$$

On the lattice $h\mathbb{Z}^d$, the definitions of the Fourier and the inverse Fourier transforms are reversed to those on a periodic box. For $f\in L_h^1$, its discrete Fourier transform is defined by
$$(\mathcal{F}_h f)(\xi):=h^d\sum_{x\in \mathbb{Z}_h^d} f(x)e^{-ix\cdot\xi}$$
on the periodic box $\frac{2\pi}{h}\mathbb{T}^d=[-\frac{\pi}{h},\frac{\pi}{h})^d$. On the other hand, the discrete inverse Fourier transform of $f\in L^1(\frac{2\pi}{h}\mathbb{T}^d)$ is defined by
$$(\mathcal{F}_h^{-1} f)(x)=\frac{1}{(2\pi)^d}\int_{h\mathbb{T}^d} f(\xi) e^{ix\cdot\xi}d\xi$$
on the lattice $h\mathbb{Z}^d$. For rapidly decreasing functions $f$ and $g$ on $h\mathbb{Z}^d$, the Plancherel theorem
$$h^d\sum_{x\in h\mathbb{Z}^d} f(x)\overline{g(x)}=\frac{1}{(2\pi)^d}\int_{\frac{2\pi}{h}\mathbb{T}^d}(\mathcal{F}_hf)(\xi)\overline{(\mathcal{F}_hg)(\xi)}d\xi$$
holds. Then, by the standard duality argument, both the discrete Fourier and the discrete inverse Fourier transforms are extended to $L_h^2$ (respectively, $L^2$)-functions.

\subsection{Sobolev spaces and the Littlewood-Paley theory}
On the lattice $h\mathbb{Z}^d$, the homogeneous differential operator $|\nabla_h|^s$ is defined by $\mathcal{F}_h\left(|\nabla_h|^sf\right)(\xi)=|\xi|^s(\mathcal{F}_hf)(\xi)$ on $\frac{2\pi}{h}\mathbb{T}^d$, while the inhomogeneous differential operator $\langle\nabla_h\rangle^s$ is defined by $\mathcal{F}_h(\langle\nabla_h\rangle^sf)(\xi)=(1+|\xi|^2)^{\frac{s}{2}}(\mathcal{F}_hf)(\xi)$. The Sobolev space $W_h^{s,p}$ (respectively, $\dot{W}_h^{s,p}$) is defined as the Banach space equipped with the norm
\begin{equation}\label{def:W}
\|f\|_{W_h^{s,p}} := \|\langle\nabla_h\rangle^sf\|_{L_h^p}\quad\left(\textup{respectively, }\|f\|_{\dot W_h^{s,p}} := \||\nabla_h|^s f\|_{L_h^p}\right).
\end{equation}
In particular, when $p=2$, we denote
$$H_h^s:=W_h^{s,2}\textup{ and }\dot{H}_h^s:=\dot{W}_h^{s,2}.$$
Indeed, there are several other natural ways to define the Sobolev spaces. In \cite{HY}, developing the Calderon-Zygmund theory on a lattice, they are shown to be equivalent.
\begin{prop}[Norm equivalence \cite{HY}]\label{norm equivalence}
	For any $1<p<\infty$, we have
	$$\|f\|_{\dot{W}_h^{s,p}}\sim \|(-\Delta_h)^{\frac{s}{2}}f\|_{L_h^p}\quad\forall s\in\mathbb{R}$$
	and
	$$\|f\|_{\dot{W}_h^{1,p}}\sim \sum_{j=1}^d\|D_{j;h}^+f\|_{L_h^p},$$
	where $D_{j;h}^+f(x):=\frac{f(x+he_j)-f(x)}{h}$.
\end{prop}
Contrary to the whole space $\mathbb{R}^d$, differential operators on the lattice $h\mathbb{Z}^d$ are bounded operators. A high Sobolev norm is bounded by a lower one even though the implicit constant blows up as $h\to0$.
\begin{lem}\label{H1 bound depending on h}
Let $h>0$ and $0\leq s\leq 1$. Then, $\|f\|_{\dot H_h^1} \lesssim h^{-(1-s)} \|f\|_{\dot H_h^s}$.
\end{lem}
\begin{proof} By the Plancherel theorem, we prove that
$$\|f\|_{\dot H_h^1} = \frac{1}{(2\pi)^{d/2}}\| |\xi| (\mathcal{F}_hf) \|_{L^2(\frac{2\pi}{h}\mathbb{T}^d)}\lesssim  \frac{h^{-(1-s)}}{(2\pi)^{d/2}}\||\xi|^s (\mathcal{F}_hf) \|_{L^2(\frac{2\pi}{h}\mathbb{T}^d)}=h^{-(1-s)} \|f\|_{\dot H_h^s},$$
where in the inequality, we used that $\xi\in\frac{2\pi}{h}\mathbb{T}^d$.
\end{proof}

Let $\phi:\mathbb{R}^d\to[0,1]$ be a radially symmetric smooth bump function such that $\phi(\xi)=1$ for $|\xi|\leq 1$ but $\phi(\xi)= 0$ for $|\xi|\geq 2$, and let $\psi:=\phi-\phi(\frac{\cdot}{2})$. 
For a dyadic number $N\in 2^{\mathbb{Z}}$ with $N\leq 1$, we define the Littlewood-Paley projection operator $P_N=P_{N;h}$ as the Fourier multiplier operator given by
\begin{equation}\label{LP}
\mathcal{F}_h(P_{N}f)(\xi)=\psi(\tfrac{2\pi h\xi}{N})(\mathcal{F}_hf)(\xi).
\end{equation}
Here, with abuse of notation, $\psi(\frac{2\pi h\xi}{N})$ denotes the function $\psi(\frac{2\pi h\xi}{N})$ restricted to the frequency domain $\frac{2\pi}{h}\mathbb{T}^d$. Then, $\sum_{N\leq 1}P_N=1$, because $\sum_{N\leq 1}\psi(\frac{2\pi\xi}{Nh})\equiv 1$ on $\frac{2\pi}{h}\mathbb{T}^d$.

The following Litttlewood-Paley inequalities are useful in our analysis in that it allows us to handle different frequencies separately.
\begin{prop}[Littlewood-Paley inequalities \cite{HY}]\label{LP inequalities}
For $1<p<\infty$, we have
$$\|f\|_{L_h^p}\lesssim\left\|\bigg\{\sum_{N\leq 1}|P_N f|^2\bigg\}^{1/2}\right\|_{L_h^p}\lesssim \|f\|_{L_h^p}.$$
\end{prop}

As applications, one can derive the Gagliardo-Nirenberg and the Sobolev inequalities \cite{Chat, KLS, HY}.

\begin{prop}\label{GN and Sobolev}\cite{HY}
Let $h>0$. Suppose that $1\leq p<q\leq\infty$ and $s>0$.
\begin{enumerate}[(i)]
\item (Gagliardo-Nirenberg inequality)
	If $\frac{1}{q}=\frac{1}{p}-\frac{\theta s}{d}$ and $0<\theta<1$, then
	$$\|f\|_{L_h^q}\lesssim \|f\|_{L_h^p}^{1-\theta}\|f\|_{\dot{W}_h^{s,p}}^\theta.$$
\item (Sobolev inequality)
	If $\frac{1}{q}=\frac{1}{p}-\frac{s}{d}$ and $q<\infty$, then
	$$\|f\|_{L_h^q}\lesssim \|f\|_{W_h^{s,p}}.$$
	If $q=\infty$ and $s>\frac{d}{p}$, then   
$$\|f\|_{L_h^\infty} \lesssim \|f\|_{W_h^{s,p}}.$$
\end{enumerate}
\end{prop}

\section{Uniform Strichartz Estimates (Proof of Theorem \ref{Strichartz})}

As mentioned in the introduction, uniform Strichartz estimates for the discrete linear Schr\"odinger equation (Theorem \ref{Strichartz}) will play a crucial role in our analysis. When $\alpha=1$, such uniform Strichartz estimates have been established in our previous work \cite[Theorem 1.3]{HY}. In this section, adapting the strategy and the harmonic analysis tools in \cite{HY} to the fractional Schr\"odinger case, we establish uniform Strichartz estimates which are not covered in the previous result, that is, the case $d=1$, $0<\alpha<1$ and $\alpha\neq\frac{1}{2}$. Indeed, the desired uniform Strichartz estimates follow from the dispersive estimates.

\begin{prop}[Frequency localized dispersive estimates]\label{dispersive estimate for schrodinger}
Suppose that $d=1$, $0<\alpha<1$ and $\alpha\neq\frac{1}{2}$. For each dyadic number $N\in 2^{\mathbb{Z}}$ with $N\leq 1$, let $P_N$ be the Littlewood-Paley projection given in \eqref{LP}. Then, the following hold.
\begin{enumerate}[(i)]
\item (Non-resonance case) If $0<\alpha<\frac{1}{2}$, then 
$$\left\|e^{-it(-\Delta_h)^\alpha} P_N f\right\|_{L_h^\infty} \lesssim \left(\frac Nh\right)^{1-\alpha}\frac{1}{|t|^{1/2}} \|f\|_{L_h^1}.$$
\item (Resonance case)  If $\frac{1}{2}<\alpha<1$, then 
$$\left\|e^{-it(-\Delta_h)^\alpha} P_N f\right\|_{L_h^\infty} \lesssim \left(\frac{N}{h}\right)^{1-\frac{2\alpha}{3}} \frac{1}{|t|^{1/3}} \|f\|_{L_h^1}.$$
\end{enumerate}
\end{prop}

\begin{proof}[Proof of Theorem \ref{Strichartz} when $0<\alpha<1$, assuming Proposition \ref{dispersive estimate for schrodinger}]
Let $\tilde{P}_N$ be the Fourier multiplier operator with symbol $\tilde{\psi}(\frac{h\xi}{N})$, where $\tilde{\psi}\in C_c^\infty$, $\tilde{\psi}\equiv 1$ on $\supp\psi$, and $\psi$ is the smooth cutoff in the definition of $P_N$. Indeed, the proof of Proposition \ref{dispersive estimate for schrodinger} does not rely on a particular choice of a frequency cut-off in $P_N$ (see below), and thus Proposition \ref{dispersive estimate for schrodinger} holds with another projection operator $\tilde{P}_N$.

Suppose that $\frac{1}{2}<\alpha<1$ so that
$$\left\|e^{-it(-\Delta_h)^\alpha} \tilde{P}_N f\right\|_{L_h^\infty} \lesssim \left(\frac{N}{h}\right)^{1-\frac{2\alpha}{3}} \frac{1}{|t|^{1/3}} \|f\|_{L_h^1}.$$
Then, it follows from the interpolation argument in Keel-Tao \cite{KT} that
$$\left\|e^{-it(-\Delta_h)^\alpha} \tilde{P}_N f\right\|_{L_t^q(\mathbb{R};L_h^r)} \lesssim \left(\frac Nh\right)^{(1-\frac{2\alpha}{3})(\frac{1}{2}-\frac{1}{r})}\|f\|_{L_h^2}= \left(\frac Nh\right)^{\frac{3-2\alpha}{q}}\|f\|_{L_h^2}$$
for all resonance admissible pairs. Inserting $P_Nf$ with $P_N=\tilde{P}_NP_N$, we get
$$\left\|e^{-it(-\Delta_h)^\alpha} P_N f\right\|_{L_t^q(\mathbb{R};L_h^r)}\lesssim \left(\frac Nh\right)^{\frac{3-2\alpha}{q}}\|P_Nf\|_{L_h^2}\lesssim\|P_N(|\nabla_h|^{\frac{3-2\alpha}{q}}f)\|_{L_h^2}.$$
Therefore, by the Littlewood-Paley inequalities (Proposition \ref{LP inequalities}), we show that
\begin{align*}
\|e^{-it(-\Delta_h)^\alpha} f\|_{L_t^q(\mathbb{R};L_h^r)}^2&\lesssim \left\|\bigg\{\sum_{N\leq 1}|e^{-it(-\Delta_h)^\alpha} \tilde{P}_N f|^2\bigg\}^{1/2}\right\|_{L_t^q(\mathbb{R};L_h^r)}^2\\
&\leq \sum_{N\leq 1}\|e^{-it(-\Delta_h)^\alpha} \tilde{P}_N f\|_{L_t^q(\mathbb{R};L_h^r)}^2\\
&\lesssim \sum_{N\leq 1}\|P_N(|\nabla_h|^{\frac{3-2\alpha}{q}}f)\|_{L_h^2}^2\lesssim \||\nabla_h|^{\frac{3-2\alpha}{q}}f\|_{L_h^2}^2.
\end{align*}
If $0<\alpha<\frac{1}{2}$, repeating but using Proposition \ref{dispersive estimate for schrodinger} $(ii)$, one can show Theorem \ref{Strichartz} $(ii)$.
\end{proof}

\begin{proof}[Proof of Proposition~\ref{dispersive estimate for schrodinger}]
	By the discrete Fourier and the discrete inverse Fourier transforms, we write
$$	\begin{aligned} 
	&\left(e^{-it(-\Delta_h)^\alpha} P_Nf\right)(x_m)\\
	&=\frac{1}{2\pi}\int_{-\frac{\pi}{h}}^{\frac{\pi}{h}} e^{-it\left\{\frac{4}{h^{2}}\sin^{2}\left(\frac{h\xi}{2}\right)\right\}^\alpha}e^{ix_m\xi}\psi(\tfrac{h\xi}{N}) \left\{h\sum_{y_m\in h\mathbb{Z}} f(y_m)e^{-iy_m\xi}\right\}d\xi \\
		&=h\sum_{y_m \in h\Z} K_{N,t}(x_m-y_m)f(y_m),
	\end{aligned}$$
	where
	$$K_{N,t}(x_m)=\frac{1}{2\pi}
	\int_{-\frac{\pi}{h}}^{\frac{\pi}{h}} e^{i\varphi_{t,x_m}(\xi)}\psi(\tfrac{h\xi}{N})d\xi$$
	and
$$	\varphi_{t,x_m}(\xi)= x_m\xi -t\left\{\tfrac{4}{h^{2}}\sin^{2}\left(\tfrac{h\xi}{2}\right)\right\}^\alpha=x_m\xi -\tfrac{4^\alpha t}{h^{2\alpha}}\sin^{2\alpha}\left(\tfrac{h\xi}{2}\right).$$
	Hence, we have
	$$\|e^{-it(-\Delta_h)^\alpha} P_N f\|_{L_h^\infty}\leq \|K_{N,t}\|_{L_h^\infty}\|f\|_{L_h^1},$$
and thus it is enough to estimate the oscillatory integral $K_{N,t}$.

We observe that the second derivative of the phase function in $K_{N,t}$ is given by
$$	\begin{aligned}
	\varphi_{t,x_m}''(\xi)&= -\frac{4^{\alpha} t}{h^{2\alpha}}2\alpha\left\{(2\alpha-1)\sin^{2\alpha-2}\left(\tfrac{h\xi}{2}\right)\cos^2\left(\tfrac{h\xi}{2}\right)-\sin^{2\alpha}\left(\tfrac{h\xi}{2}\right)\right\}\frac{h^2}{4}\\
	&= -\frac{2\alpha4^{\alpha-1} th^{2-2\alpha}}{\sin^{2-2\alpha}\left(\frac{h\xi}{2}\right)}\left\{(2\alpha-1)\cos^2\left(\tfrac{h\xi}{2}\right)-\sin^{2}\left(\tfrac{h\xi}{2}\right)\right\}\\
	&= \frac{2\alpha4^{\alpha-1} th^{2-2\alpha}}{\sin^{2-2\alpha}\left(\frac{h\xi}{2}\right)}\Big\{1-\alpha-\alpha\cos(h\xi)\Big\},
	\end{aligned}$$	
and that $\varphi_{t,x_m}''(\xi_0)=0$ if and only if $\cos(h\xi_0)=\frac{1-\alpha}{\alpha}$. If $0<\alpha<\frac{1}{2}$, then there is no $\xi_0$ such that $\varphi_{t,x_m}''(\xi_0)=0$. Moreover, the second derivative of the phase function satisfies the lower bound $|\varphi_{t,x_m}''(\xi)|\gtrsim  |t|\frac{h^{2-2\alpha}}{N^{2-2\alpha}}$ on $\supp\psi(\frac{h\cdot}{N})$, because $\frac{h|\xi|}{2\pi}\leq |\sin(\frac{h\xi}{2})|\leq\frac{h|\xi|}{2}$. Therefore, van der Corput's lemma implies that $|K_{N,t}(x_m)|\lesssim (\frac{N}{h})^{1-\alpha}|t|^{-1/2}$.
	
	If $\frac{1}{2}<\alpha<1$, then there exists a unique $\xi_0>0$ such that $\varphi_{t,x_m}''(\pm\xi_0)=0$. However, in the case $N$ is too small, $\xi_0$ is not in $\supp\psi(\frac{h\cdot}{N})$, because $\cos\theta\to 1$ as $\theta\to 0$. In this case, one can get  the bound $|K_{N,t}(x_m)|\lesssim (\frac{N}{h})^{1-\alpha}|t|^{-1/2}$ as above. Then, interpolating with the trivial bound $|K_{N,t}(x_m) |\lesssim\frac{N}{h}$, we prove that $|K_{N,t}(x_m) |\lesssim  (\frac{N}{h})^{1-\frac{2\alpha}{3}}|t|^{-1/3}$. Suppose that $N$ is not too small, in other words, $N\geq N_0$ for some dyadic number $N_0>0$. Then, differentiating the phase function once more, we observe that 
$$	|\varphi_{t,x_m}'''(\pm\xi_0)|=\left|\frac{2\alpha^24^{\alpha-1} th^{3-2\alpha}}{\sin^{2-2\alpha}\left(\frac{\pm h\xi_0}{2}\right)}\sin(\pm h\xi_0)\right|\gtrsim_{N_0} |t| h^{3-2\alpha}$$
if $\xi_0\in \supp\psi(\frac{h\cdot}{N})$. Hence, by continuity, we have that on $\supp\psi(\frac{h\cdot}{N})$, either
	$|\varphi_{t,x_m}''(\xi)|\gtrsim |t|\frac{h^{2-2\alpha}}{N^{2-2\alpha}}$ or $|\varphi_{t,x_m}'''(\xi)|\gtrsim |t| h^{3-2\alpha}$. Therefore, by Van der Corput's lemma again, we prove that
$$|K_{N,t}(x_m)|\lesssim_{N_0} \max\left\{\left(\tfrac{N}{h}\right)^{1-\alpha}|t|^{-\frac{1}{2}}, \left(\tfrac{N}{h}\right)^{1-\frac{2\alpha}{3}}|t|^{-\frac{1}{3}}\right\}.$$
Finally, interpolating with the trivial bound $|K_{N,t}(x_m) |\lesssim\frac{N}{h}$ as in the previous case, we get the desired bound $|K_{N,t}(x_m) |\lesssim  (\frac{N}{h})^{1-\frac{2\alpha}{3}}|t|^{-1/3}$. 
\end{proof}

Combining the uniform Sobolev inequality (Proposition \ref{GN and Sobolev} $(ii)$) and the uniform Strichartz estimates (Theorem \ref{Strichartz}), we deduce a time-averaged uniform $L_h^\infty$-bound on discrete linear Schr\"odinger flows.
\begin{cor}[Uniform $L_h^\infty$-bounds on discrete linear Schr\"odinger flows]\label{linear L^infty bound}
Suppose that $d=1$, $0<\alpha<1$ and $\alpha\neq\frac{1}{2}$ or that $d=2,3$ and $\alpha=1$. For sufficiently small $\delta>0$, let
\begin{equation}\label{q*}
q_*=\left\{\begin{aligned}
&\infty&&\textup{if }d=1\textup{ and }\tfrac{1}{2}<\alpha\leq 1,\\
&\tfrac{4\alpha}{1-2\alpha+\delta}&&\textup{if }d=1\textup{ and }\tfrac{1}{3}<\alpha<\tfrac{1}{2},\\
&\tfrac{4}{d-2+\delta}&&\textup{if }d=2,3\textup{ and }\alpha=1.
\end{aligned}\right.
\end{equation}
Then, 
$$\|e^{-it(-\Delta_h)^\alpha}f\|_{L_t^{q_*}(\mathbb{R}; L_h^\infty)}\lesssim \|f\|_{H_h^\alpha}.$$
\end{cor}

\begin{proof}
When $d=1$ and $\frac{1}{2}<\alpha<1$, the corollary immediately follows from the uniform Sobolev embedding $H_h^\alpha\hookrightarrow L_h^\infty$(Proposition \ref{GN and Sobolev} $(ii)$) and unitarity of the linear propagator on $H_h^\alpha$. If $d=1$ and $\frac{1}{3}<\alpha<\frac{1}{2}$, then we apply the Sobolev inequality and Strichartz estimates (Theorem \ref{Strichartz}) to get
\begin{align*}
\|e^{-it(-\Delta_h)^\alpha}f\|_{L_t^{\frac{4\alpha}{1-2\alpha+\delta}}(\mathbb{R}; L_h^\infty)}&\lesssim \|e^{-it(-\Delta_h)^\alpha}f\|_{L_t^{\frac{4\alpha}{1-2\alpha+\delta}}(\mathbb{R}; W_h^{\frac{3\alpha-1-(1-\alpha)\delta}{2\alpha}, \frac{2\alpha}{3\alpha-1-\delta}})}\\
&\lesssim \|f\|_{H_h^\alpha}
\end{align*}
for sufficiently small $\delta>0$. By the same way, one can show the desired inequality in the case $d=2,3$ and $\alpha=1$.
\end{proof}

\begin{rem}\label{high d restriction}
High dimensions $d\geq 4$ are not covered in this paper, because $q_*=\tfrac{4}{d-2+\delta}$ in Corollary \ref{linear L^infty bound} is required to be $\geq 2$.
\end{rem}
\section{Uniform $L_h^\infty$-bound for Nonlinear Solutions}

We consider solutions to DNLS \eqref{DNLS}. It turns out that they exist globally in time, are unique, and obey the mass and the energy conservation laws.

\begin{prop}[Global well-posedness]\label{GWP}
Suppose that $d\in\mathbb{N}$, $\alpha\in\mathbb{R}$ and $p>1$. Then, for any initial data $u_{h,0}\in L_h^2$, there exists a unique global solution $u_h(t)\in C(\mathbb{R}; L_h^2)$ to DNLS \eqref{DNLS}. Moreover, it conserves the mass
$$M_h(u_h):=\|u_h\|_{L_h^2}^2$$
and the energy
$$E_h(u_h):=\frac{1}{2}\|(-\Delta_h)^{\frac{\alpha}{2}}u_h\|_{L_h^2}^2+\frac{\lambda}{p}\|u_h\|_{L_h^p}^p.$$
\end{prop}

\begin{proof}
The proof follows as in \cite[Proposition]{HY}, where the case $\alpha=1$ is considered. Indeed, for fixed $h>0$, local well-posedness in $L_h^2$ (as well as the conservation laws) can be proved by the trivial inequality $\|u_h\|_{L_h^\infty}\leq h^{-d/2}\|u_h\|_{L_h^2}$ on a short time interval depending on $h>0$. The interval of existence is then extended by the mass conservation law.
\end{proof}

As observed in the previous section, discrete linear Schr\"odinger flows satisfy a time-averaged uniform $L_h^\infty$-bound (Corollary \ref{linear L^infty bound}). The purpose of this section is to show that nonlinear flows obey a similar bound at least locally in time.

\begin{prop}[Uniform $L_h^\infty$-bound for nonlinear solutions]\label{L^infty bound}
Suppose that $d$, $\alpha$, $p$ and $\lambda$ satisfy the hypotheses in Theorem \ref{main theorem}, and let $q_*$ be given by \eqref{q*}. Then, the global solution $u_h(t)$ to DNLS  \eqref{DNLS} with initial data $u_{h,0}\in H_h^\alpha$, constructed in Proposition \ref{GWP}, satisfies \begin{equation}\label{long time L^infty bound}
\|u_h\|_{L_t^{q_*}([-T,T]; L_h^\infty)}\lesssim \langle T\rangle^{1/q_*}\|u_{h,0}\|_{H_h^\alpha},\quad\forall T>0.
\end{equation}
\end{prop}

The proof is similar to that of local well-posedness of the continuum equation \eqref{NLS}, but the following nonlinear estimate is employed.

\begin{lem}[Nonlinear estimate]\label{Lem:Nonlinear estimate}
	Suppose that $p>1$ and $0\le s\le1$. Then,
	\begin{equation}\label{Ineq:Nonlinear term}
\left\||\nabla_h|^s(|u_h|^{p-1}u_h) \right\|_{L_h^2}\leq C  \| u_h \|_{L_h^\infty}^{p-1} \| |\nabla_h|^s u_h \|_{L_h^2}.
	\end{equation}
\end{lem}

\begin{proof}
For the case $s=1$, we refer \cite[Lemma~6.2]{HY}. Thus, we may only consider the fractional case $0<s<1$. We claim that if $0<s<1$, then
$$\| |\nabla_h|^s f \|_{L_h^2}^2 \sim  h^d\sum_{y_m\in h\Z^d, y_m\neq 0} \frac{\left\|f(\cdot+y_m)-f\right\|_{L_h^2}^2}{|y_m|^{d+2s}},$$
	where the implicit constants are not independent of $h>0$. Indeed, by the Plancherel theorem, we may write the right hand side as
	\begin{align*}
	&\frac{h^d}{(2\pi)^d}\sum_{y_m\in h\Z^d, y_m\neq 0} \frac{\| ( e^{-i y_m\cdot \xi} -1) \mathcal{F}_hf (\xi)\|_{L_\xi^2(\frac{2\pi}{h}\mathbb{T}^d)}^2}{|y_m|^{d+2s}}\\
	&= \frac{1}{(2\pi)^d}\int_{\frac{2\pi}{h}\mathbb{T}^d}\left\{h^d\sum_{y_m\in h\Z^d, y_m\neq 0} \frac{| e^{-i y_m\cdot \xi} -1 |^2 }{ |y_m|^{d+2s}}\right\}|\mathcal{F}_hf(\xi)|^2 d\xi.
	\end{align*}
In consideration of the Riemann integration and by changes of variables, we get
	$$\begin{aligned}
	h^d\sum_{y_m\in h\Z^d, y_m\neq 0} \frac{| e^{-i y_m\cdot \xi} -1 |^2 }{ |y_m|^{d+2s}}\underset{h\to0}\longrightarrow& \int_{\mathbb{R}^d} \frac{| e^{-i y\cdot \xi} -1 |^2 }{ |y|^{d+2s}}dy= |\xi|^s\int_{\mathbb{R}^d} \frac{| e^{-i y\cdot \frac{\xi}{|\xi|}} -1 |^2 }{ |y|^{d+2s}}dy\\
	&=|\xi|^s\int_{\mathbb{R}}\cdots\int_{\mathbb{R}} \frac{| e^{-i y_1} -1 |^2 }{ |y|^{d+2s}}dy_1\cdots dy_d\sim |\xi|^s.
	\end{aligned}$$
	Thus, the claim follows.
	
	To show the lemma, we observe that by the fundamental theorem of calculus,
	\begin{equation}\label{FTC application}
	\begin{aligned}
	|u|^{p-1}u-|v|^{p-1}v&=\int_0^1\frac{d}{ds}\left(|su+(1-s)v|^{p-1}(su+(1-s)v)\right)ds\\
	&=\frac{p+1}{2}\left\{\int_0^1\left|su+(1-s)v\right|^{p-1}ds\right\} (u-v)\\
	&\quad+\frac{p-1}{2}\left\{\int_0^1\left|su+(1-s)v\right|^{p-3}\left(su+(1-s)v\right)^2ds\right\} \overline{u-v}.
	\end{aligned}
	\end{equation}
	Hence, by the claim and \eqref{FTC application}, we prove that
	\begin{align*}
	\left\||\nabla_h|^s(|u_h|^{p-1}u_h) \right\|_{L_h^2}^2&\sim h^d\sum_{y_m\in h\Z^d, y_m\neq 0} \frac{\left\||u_h|^{p-1}u_h(\cdot+y_m)-|u_h|^{p-1}u_h\right\|_{L_h^2}^2}{|y_m|^{d+2s}}\\
	&\leq 2p\|u_h\|_{L_h^\infty}^{2(p-1)}\cdot h^d\sum_{y_m\in h\Z^d, y_m\neq 0} \frac{\left\|u_h(\cdot+y_m)-u_h\right\|_{L_h^2}^2}{|y_m|^{d+2s}}\\
	&\sim \|u_h\|_{L_h^\infty}^{2(p-1)}\||\nabla_h|^su_h\|_{L_h^2}^2.
	\end{align*}
\end{proof}

\begin{proof}[Proof of Proposition \ref{L^infty bound}]
First, we will show that the proposition holds on a sufficiently small interval. Let $I=[-\tau,\tau]$ be a sufficiently small interval to be chosen later, and define
$$\Gamma(u_h):=e^{-it(-\Delta_h)^\alpha}u_{h,0}-i\lambda\int_0^t e^{-i(t-s)(-\Delta_h)^\alpha}\left(|u_h|^{p-1}u_h\right)(s)ds$$
on the set
$$X:=\Bigg\{u_h:h\mathbb{Z}^d\to\mathbb{C}:\ \|u_h\|_{C_t(I;H_h^\alpha)}+\|u_h\|_{L_t^{q_*}(I;L_h^\infty)}\leq 2(1+c)\|u_{h,0}\|_{H_h^\alpha}\Bigg\}$$
equipped with the norm $\|\cdot\|_{C_t(I; L_h^2)}$. Here, $c>0$ denotes the uniform-in-$h$ constant in Corollary \ref{linear L^infty bound}, that is,
$$\|e^{-it(-\Delta_h)^\alpha}u_{h,0}\|_{L_t^{q_*}(\mathbb{R}; L_h^\infty)}\leq c\|u_{h,0}\|_{H_h^\alpha}.$$
Note that $X$ is a complete metric space. Indeed, if $\{v_{h,n}\}_{n=1}^\infty$ is a Cauchy sequence in $X$, then it converges to $v_h$ in $C_t(I;L_h^2)$ as $n\to \infty$. However, since $\|u_h\|_{C_t(I;H_h^\alpha)}+\|u_h\|_{L_t^{q_*}(I;L_h^\infty)}\lesssim_h \|u_h\|_{C_t(I;L_h^2)}$, we have $v_{h,n}\to v_h$ in $C_t(I;H_h^\alpha)\cap L_t^{q_*}(I;L_h^\infty)$, and so $v_h\in X$.

We claim that $\Gamma$ is contractive on $X$. Indeed, by unitarity, we have
$$\|\Gamma(u_h)\|_{C_t(I; H_h^\alpha)}\leq \|u_{h,0}\|_{H_h^\alpha}+|\lambda|\||u_h|^{p-1}u_h\|_{L_t^1(I;H_h^\alpha)}$$
and by Corollary \ref{linear L^infty bound},
$$\|\Gamma(u_h)\|_{L_t^{q_*}(I; L_h^\infty)}\leq c\|u_{h,0}\|_{H_h^\alpha}+c|\lambda|\||u_h|^{p-1}u_h\|_{L_t^1(I;H_h^\alpha)}.$$
Thus, applying the nonlinear estimate (Lemma \ref{Lem:Nonlinear estimate}) and the H\"older inequality in time, we show that if $u_h\in X$, 
\begin{align*}
&\|\Gamma(u_h)\|_{C_t(I; H_h^\alpha)}+\|\Gamma(u_h)\|_{L_t^{q_*}(I; L_h^\infty)}\\
&\leq (1+c)\|u_{h,0}\|_{H_h^\alpha}+(1+c)|\lambda|(2\tau)^{1-\frac{p-1}{q_*}}C\|u_h\|_{L_t^{q_*}(I;L_h^\infty)}^{p-1}\|u_h\|_{C_t(I;H_h^\alpha)}\\
&\leq (1+c)\|u_{h,0}\|_{H_h^\alpha}+(1+c)|\lambda|(2\tau)^{1-\frac{p-1}{q_*}}C\left( 2(1+c)\|u_{h,0}\|_{H_h^\alpha}\right)^p\\
&=(1+c)\|u_{h,0}\|_{H_h^\alpha}\left\{1+(1+c)|\lambda|(2\tau)^{1-\frac{p-1}{q_*}}C\left( 2(1+c)\|u_{h,0}\|_{H_h^\alpha}\right)^{p-1}\right\}.
\end{align*}
On the other hand, by \eqref{FTC application}, we obtain
\begin{align*}
\|\Gamma(u_h)-\Gamma(v_h)\|_{C_t(I; L_h^2)}&=|\lambda|\left\|\int_0^t e^{-i(t-s)(-\Delta_h)^\alpha}\left(|u_h|^{p-1}u_h-|v_h|^{p-1}v_h\right)(s)ds\right\|_{C_t(I; L_h^2)}\\
&\leq |\lambda|\||u_h|^{p-1}u_h- |v_h|^{p-1}v_h\|_{L_t^1(I;L_h^2)}\\
&\leq |\lambda|(2\tau)^{1-\frac{p-1}{q_*}}p\left(\|u_h\|_{L_t^{q_*}(I;L_h^\infty)}+\|v_h\|_{L_t^{q_*}(I;L_h^\infty)}\right)^{p-1}\\
&\quad\cdot\|u_h-v_h\|_{C_t(I;L_h^2)}\\
&\leq |\lambda|(2\tau)^{1-\frac{p-1}{q_*}}\left(2c\|u_{h,0}\|_{H_h^\alpha}\right)^{p-1}\|u_h-v_h\|_{C_t(I;L_h^2)},
\end{align*}
provided that $u_h, v_h\in X$. Then, since $q_*$ is selected so that $1-\frac{p-1}{q_*}>0$ (see \eqref{q*}), one can make $\Gamma$ contractive on $X$ taking sufficiently small $\tau>0$ depending on $\|u_{h,0}\|_{H_h^\alpha}$ but not on $h>0$.

By the claim, it follows from the Banach fixed point theorem that $\Gamma$ has a fixed point in $X$, which is, by uniqueness (Proposition \ref{GWP}), the solution $u_h$ to DNLS \eqref{DNLS}. Therefore, we conclude that
\begin{equation}\label{short time L^infty bound}
\|u_h\|_{L_t^q([-\tau,\tau]; L_h^\infty)}\lesssim \|u_{h,0}\|_{H_h^\alpha}.
\end{equation}

In order to extend the time interval arbitrarily, we show that $\|u_h(t)\|_{H_h^\alpha}$ is bounded globally-in-time. Indeed, by the mass conservation law and the norm equivalence (Proposition \ref{norm equivalence}), it is enough to show that $\|(-\Delta_h)^{\frac{\alpha}{2}}u_h(t)\|_{L_h^2}$ is bounded globally in time. When $\lambda>0$, the energy conservation law implies that $\|(-\Delta_h)^{\frac{\alpha}{2}}u_h(t)\|_{L_h^2}^2\leq 2E_h(u_h(t))=2E_h(u_{h,0})$ for all $t$. When $\lambda<0$, we use both the mass and the energy conservation laws as well as the uniform Gagliardo-Nirenberg inequality (Proposition \ref{GN and Sobolev} (i)) to get
\begin{equation}\label{conservation laws and GN applications}
\begin{aligned}
\frac{1}{2}\|(-\Delta_h)^{\frac{\alpha}{2}}u_h(t)\|_{L_h^2}^2&=E_h(u_h(t))+\frac{\lambda}{p+1}\|u_{h}(t)\|_{L_h^{p+1}}^{p+1}\\
&\leq E_h(u_h(t))+C\|u_h(t)\|_{L_h^2}^{p+1-\frac{d(p-1)}{2\alpha}}\|(-\Delta_h)^{\frac{\alpha}{2}}u_h(t)\|_{L_h^2}^{\frac{d(p-1)}{2\alpha}}\\
&\leq E_h(u_{h,0})+CM_h(u_{h,0})^{\frac{p+1}{2}-\frac{d(p-1)}{4\alpha}}\|(-\Delta_h)^{\frac{\alpha}{2}}u_h(t)\|_{L_h^2}^{\frac{d(p-1)}{2\alpha}}.
\end{aligned}
\end{equation}
By the assumption, we have $\frac{d(p-1)}{2\alpha}<2$. Thus, we can use Young's inequality to bound $\|(-\Delta_h)^{\frac{\alpha}{2}}u_h(t)\|_{L_h^2}^2$ in terms of the mass $M_h(u_{h,0})$ and the energy $E_h(u_{h,0})$.

Because $\|u_h(t)\|_{H_h^\alpha}$ is bounded uniformly in time, \eqref{short time L^infty bound} can be iterated with new initial data $u(\tau), u(2\tau), ...$ with the bounds \eqref{short time L^infty bound} on the intervals $[\tau,2\tau]$, $[2\tau, 3\tau]$, ... to cover an arbitrarily long time interval $[-T,T]$. Therefore, summing up, we obtain \eqref{long time L^infty bound}.
\end{proof}


\section{Some Properties of Discretization and Linear Interpolation}
In the previous two sections, we have discussed about time-averaged uniform $L_h^\infty$-bounds for solutions to the discrete linear and nonlinear Schr\"odinger equations. We now turn to our attentions to the continuum limit problems.

We recall that a locally integrable function $f:\mathbb{R}^d\rightarrow\C$ is discretized by
\begin{equation}\label{def: discretization}
f_h(x_m):=\frac{1}{h^d}\int_{x_m+[0,h)^d}f(x)dx,\quad \forall x_m\in h\mathbb{Z}^d.
\end{equation}
On the other hand, a function $f:h\mathbb{Z}^d\rightarrow \C$ on a lattice domain becomes continuous by linear interpolation
\begin{equation}\label{def:  linear interpolation}
(p_h f)(x):= f(x_m)+(D_h^+f)(x_m)\cdot(x-x_m),\quad\forall x\in x_m+[0,h)^d,
\end{equation}
where $D_{h}^+=(D_{h;1}^+,\cdots, D_{h;d}^+)$ is the discrete gradient operator, and
$$(D_{h;j}^+ f)(x_m):=\frac{f(x_m+he_j)-f(x_m)}{h}$$
is the discrete $j$-th partial derivative. In this section, we collect some important properties of discretization and linear interpolation, which will be conveniently used in the proof of the main theorem.

We begin with boundedness of discretization and linear interpolation (see \cite{KLS}).
\begin{lem}[Boundedness of discretization]\label{discretization inequality} 
If $f\in H^\alpha(\mathbb{R}^d)$ with $0\leq \alpha\leq 1$, then its discretization $f_h$ satisfies $\| f_h\|_{\dot{H}_h^\alpha(h\mathbb{Z}^d)}\lesssim \| f\|_{\dot{H}^\alpha(\mathbb{R}^d)}$.
\end{lem}

\begin{proof}
By (complex) interpolation, it suffices to show the lemma with $\alpha=0$ and $\alpha=1$. Indeed, by the definition \eqref{def: discretization} and the Cauchy-Schwarz inequality, we get the bound,
\begin{align*}
\| f_h \|_{L_h^2(h\mathbb{Z}^d)}^2 &=h^d\sum_{x_m\in h\mathbb{Z}^d}|f_h(x_m)|^2= h^d \sum_{x_m\in h\Z^d} \left| \frac{1}{h^d} \int_{x_m+[0,h)^d} f(x) dx \right|^2\\
&\le \sum_{x_m\in h\Z^d}  \int_{x_m+[0,h)^d} |f(x)|^2 dx=\|f\|_{L^2(\mathbb{R}^d)}^2.
\end{align*}
Similarly when $\alpha=1$, we write
\begin{align*}
\|D_{h;j}f_h\|_{L_h^2(h\mathbb{Z}^d)}^2
&= h^d\sum_{x_m\in h\Z^d}\left|\frac{f_h(x_m+h e_j)-f_h(x_m)}{h}\right|^2\\ 
&=  h^d\sum_{x_m\in h\Z^d}\left| \frac{1}{h^{d+1}}\int_{x_m+[0,h)^d}f(x+he_j)-f(x)dx \right|^2 \\
&\le \sum_{x_m\in h\Z^d}\int_{x_m+[0,h)^d}\left|\frac{f(x+he_j)-f(x)}{h}\right|^2dx.
\end{align*}
Then, it follows from the fundamental theorem of calculus
$$\frac{f(x+he_j)-f(x)}{h}=\frac{1}{h}\int_0^1\frac{d}{ds}\left(f(x+hse_j)\right)ds=\int_0^1 (\nabla_{x_j} f)(x+hse_j) ds$$
that
\begin{align*}
\|D_{h;j}^+f_h\|_{L_h^2(h\mathbb{Z}^d)}^2
&\le \sum_{x_m\in h\Z^d}\int_{x_m+[0,h)^d} \int_0^1\left|(\nabla_{x_j} f)(x+hse_j)\right|^2dsdx \\
&=\int_{\mathbb{R}^d} \int_0^1\left|(\nabla_{x_j} f)(x+hse_j)\right|^2dsdx=\|\nabla_{x_j} f\|_{L^2(\mathbb{R}^d)}^2.
\end{align*}
Thus, summing in $j$, we prove the inequality with $\alpha=1$.
\end{proof}

\begin{lem}[Boundedness of linear interpolation]\label{linear interpolation inequality}
If $0\le \alpha\le 1$, then $\| p_h f_h \|_{\dot{H}^\alpha(\mathbb{R}^d)}\lesssim \| f_h \|_{\dot{H}_h^\alpha(h\mathbb{Z}^d)}$.
\end{lem}

\begin{proof}
By (complex) interpolation again, it suffices to consider the endpoint cases $\alpha=0$ and $\alpha=1$. For $\alpha=0$, we write
\begin{align*}
\|p_h f_h\|_{L^2(\mathbb{R}^d)}^2&=\sum_{x_m\in h\mathbb{Z}^d}\int_{x_m+[0,h)^d} \Big|f_h(x_m)+\sum_{j=1}^d\frac{f_h(x_m+he_j)-f_h(x_m)}{h}(x-x_m)_j\Big|^2 dx\\
&\lesssim h^d\sum_{x_m\in h\mathbb{Z}^d}\Big\{\left|f_h(x_m)\right|^2+\sum_{j=1}^d\left|f_h(x_m+he_j)\right|^2\Big\}=(d+1)\|f_h\|_{L_h^2(h\mathbb{Z}^d)}.
\end{align*}
For $\alpha=1$, we observe that
$$\|\nabla_{x_j}(p_hf_h)\|_{L^2(x_m+[0,h)^d)}^2=\|(D_{h;j}^+f_h)(x_m)\|_{L^2(x_m+[0,h)^d)}^2=h^d\left|(D_{h;j}^+f_h)(x_m)\right|^2.$$
Thus, summing in $x_m$ and $j$, we prove the desired inequality.
\end{proof}

Next, we show that linear interpolation is almost a reverse action to discretization up to small error. Moreover, the error can be precisely estimated in a lower regularity norm.
\begin{prop}[Linear interpolation vs. discretization]\label{Prop:convergence rate of phgh}
Suppose that $f\in H^\alpha(\mathbb{R}^d)$ with $0\leq \alpha\leq1$, and let $f_h: h\mathbb{Z}^d\to\mathbb{C}$ be its discretization (see \eqref{discretization}). Then,
$$\|p_hf_h-f\|_{L^2(\R^d)} \lesssim h^\alpha\|f\|_{H^\alpha(\R^d)}.$$
\end{prop}

\begin{proof}
We claim that the proposition with $\alpha=1$, that is,
\begin{equation}\label{Prop:convergence rate of phgh, alpha=1}
\|p_hf_h-f\|_{L^2(\R^d)} \lesssim h\|f\|_{H^1(\R^d)},
\end{equation}
holds for smooth functions. Indeed, 
\begin{equation}\label{dis vs lin proof}
\|p_hf_h-f\|_{L^2(\R^d)}^2=\sum_{x_m\in h\mathbb{Z}^d} \|p_hf_h(x)-f(x)\|_{L_x^2(x_m+[0,h)^d)}^2.
\end{equation}
For each term in the sum, we observe from the definitions \eqref{def: discretization} and \eqref{def: linear interpolation} that if $x\in x_m+[0,h)^d$, then
\begin{align*}
&\left|(p_hf_h)(x)-f(x)\right|\\
&=\left|\Big\{f_h(x_m)+(D_h^+f_h)(x_m)\cdot(x-x_m)\Big\}-f(x)\right|\\
&=\Big|\frac{1}{h^{d}}\int_{x_m+[0,h)^d}\Big\{\left(f(y)-f(x)\right)+\sum_{j=1}^d\left(f(y+he_j)-f(y)\right)\frac{(x-x_m)_j}{h}\Big\} dy\Big|\\
&\lesssim \frac{1}{h^{\frac{d}{2}}}\Big\{\int_{|x-y|\leq \sqrt{d} h}|f(x)-f(y)|^2+\sum_{j=1}^d|f(y+he_j)-f(y)|^2 dy\Big\}^{\frac{1}{2}},
\end{align*}
where the Cauchy-Schwarz inequality is used in the last step. Hence, inserting this bound in \eqref{dis vs lin proof}, we get
$$\|p_hf_h-f\|_{L^2(\R^d)}^2\lesssim \frac{1}{h^d}\int_{\mathbb{R}^d}\int_{|x-y|\leq\sqrt{d} h}|f(x)-f(y)|^2+\sum_{j=1}^d|f(y+he_j)-f(y)|^2dydx.$$
Then, applying the fundamental theorem of calculus
$$f(x)-f(y)=\left\{\int_0^1 \nabla f(y+s(x-y))ds\right\}\cdot(x-y),$$
we show that
\begin{align*}
&\|p_hf_h-f\|_{L^2(\R^d)}^2\\
&\lesssim h^{2-d}\int_0^1\int_{\mathbb{R}^d}\int_{|x-y|\leq\sqrt{d} h}|\nabla f(y+s(x-y))|^2+\sum_{j=1}^d|(\nabla_{x_j} f)(y+hse_j)|^2dxdyds\\
&=h^{2-d}\int_0^1\int_{|x|\leq\sqrt{d} h}\int_{\mathbb{R}^d}|\nabla f(y+sx)|^2+\sum_{j=1}^d|(\nabla_{x_j} f)(y+hse_j)|^2dydxds\\
&\sim h^2\|\nabla f\|_{L^2(\mathbb{R}^d)}^2.
\end{align*}

If $f$ is just in $H^\alpha(\mathbb{R}^d)$, then we decompose
$$f=f_{low}+f_{high},$$
where $\widehat{f_{low}}(\xi)=\mathbf{1}_{[-\frac{\pi}{h},\frac{\pi}{h})^d}(\xi)\hat{f}(\xi)$ and $f_{high}=f-f_{low}$. For the low frequency part $f_{low}$, which is smooth, we apply \eqref{Prop:convergence rate of phgh, alpha=1} to get
$$\|p_h(f_{low})_h-f_{low}\|_{L^2(\R^d)} \lesssim h\|f_{low}\|_{H^1(\R^d)}\lesssim h^\alpha \|f\|_{H^\alpha(\mathbb{R}^d)}.$$
On the other hand, for the high frequency part, by trivial estimates and Lemma \ref{linear interpolation inequality} and \ref{discretization inequality}, we have
\begin{align*}
\|p_h(f_{high})_h-f_{high}\|_{L^2(\R^d)} &\leq \|p_h(f_{high})_h\|_{L^2(\R^d)}+\|f_{high}\|_{L^2(\R^d)}\\
&\lesssim\|(f_{high})_h\|_{L_h^2(h\mathbb{Z}^d)}+\|f_{high}\|_{L^2(\R^d)}\\
&\lesssim\|f_{high}\|_{L^2(\R^d)}\lesssim h^\alpha\|f\|_{H^\alpha(\mathbb{R}^d)}.
\end{align*}
Therefore, summing up, we complete the proof.
\end{proof}

As an application of Proposition \ref{Prop:convergence rate of phgh}, we show that the linear interpolation of the discrete linear Schr\"odinger flow is almost like it continuum flow.

\begin{prop}[Linear interpolation of the discrete Schr\"odinger flow]\label{linear approx}
If $u_0\in H^\alpha(\R^d)$ and $u_{h,0}\in H_h^\alpha(h\Z^d)$ with $0\leq \alpha\leq1$, then
\begin{align*}
&\| p_h e^{-it(-\Delta_h)^\alpha} u_{h,0} - e^{-it(-\Delta)^\alpha}u_0 \|_{L^2(\R^d)}\\
&\lesssim h^{\frac{\alpha}{1+\alpha}}|t|\left\{\|u_{h,0}\|_{H_h^\alpha(h\Z^d)}+\|u_0\|_{H^\alpha(\R^d)}\right\}+\|p_hu_{h,0}-u_0\|_{L^2(\mathbb{R}^d)}.
\end{align*}
In particular, if $u_{h,0}$ is the discretization of $u_0$, then by Lemma \ref{discretization inequality} and Proposition \ref{Prop:convergence rate of phgh}, 
$$\| p_h e^{-it(-\Delta_h)^\alpha} u_{h,0} - e^{-it(-\Delta)^\alpha}u_0 \|_{L^2(\R^d)}\lesssim 
 \langle t\rangle h^{\frac{\alpha}{1+\alpha}}\|u_0\|_{H^\alpha(\R^d)}.$$
\end{prop}

For the proof, we need the following lemmas.

\begin{lem}[Symbol of the linear interpolation operator]\label{p_h symbol}
The interpolation operator $p_h$ is a Fourier multiplier operator in the sense that
$$\widehat{p_h u_h}(\xi) = \mathcal{P}_h(\xi) \mathcal (\tilde{\mathcal{F}}_h u_h)(\xi),\quad\forall\xi\in\mathbb{R}^d,$$
	where $\mathcal{P}_h:\mathbb{R}^d\to\mathbb{C}$ and $\tilde{\mathcal{F}_h}$ denotes the $[-\frac{\pi}{h},\frac{\pi}{h})^d$-periodic extension of the discrete Fourier transform $\mathcal{F}_h$, that is, $(\tilde{\mathcal{F}_h}u_h)(\xi')=(\mathcal{F}_hu_h)(\xi)$ for all $\xi'\in\xi+\frac{2\pi}{h}\mathbb{Z}^d.$
\end{lem}

\begin{proof}
We decompose the Fourier transform of $p_h u_h$ into integrals on small cubes, 
$$\widehat{p_h u_h}(\xi)=\sum_{x_m\in h\Z^d} \int_{x_m+[0,h)^d} (p_h u_h)(x)e^{-ix\cdot \xi} dx.$$
By the definition of $p_h$ (see \eqref{p_h}) and by simple changes of variables, each integral can be written explicitly as
$$\begin{aligned}
&\int_{x_m+[0,h)^d} (p_h u_h)(x)e^{-ix\cdot \xi} dx\\
&=\int_{x_m+[0,h)^d} \Bigg\{ u_h(x_m) +\sum_{j=1}^d\frac{u_h(x_m+he_j)-u_h(x_m)}{h} (x-x_m)_j\Bigg\}e^{-ix\cdot \xi} dx \\
&=u_h(x_m)e^{-ix_m\cdot\xi}\int_{[0,h)^d} e^{-ix\cdot \xi} dx+e^{-ix_m\cdot\xi}\sum_{j=1}^d\frac{u_h(x_m+he_j)-u_h(x_m)}{h} \left\{\int_{[0,h)^d} x_j e^{-ix\cdot \xi}dx\right\}.
\end{aligned}$$
Then, summing in $x_m$, we get
$$\begin{aligned}
\widehat{p_h u_h}(\xi)&=\sum_{x_m\in h\mathbb{Z}^d}u_h(x_m)e^{-ix_m\cdot\xi}\Bigg\{\int_{[0,h)^d} e^{-ix\cdot \xi} dx+\sum_{j=1}^d \frac{e^{ih\xi_j}-1}{h}\int_{[0,h)^d} x_j e^{-ix\cdot \xi}dx\Bigg\}\\
&=\mathcal{P}_h(\xi) \mathcal (\tilde{\mathcal{F}}_h u_h)(\xi),
\end{aligned}$$
where
\begin{equation}\label{p_h symbol calculation}
\mathcal{P}_h(\xi)=\frac{1}{h^d}\int_{[0,h)^d} e^{-ix\cdot \xi} dx+\sum_{j=1}^d \frac{e^{ih\xi_j}-1}{h}\cdot\frac{1}{h^d}\int_{[0,h)^d} x_j e^{-ix\cdot \xi}dx.
\end{equation}
Here, we used that $e^{-ix_m\cdot\xi}$ is $[-\frac{\pi}{h},\frac{\pi}{h})^d$-periodic, since $x_m\in h\mathbb{Z}^d$.
\end{proof}

\begin{rem}
By a direct computation from \eqref{p_h symbol calculation}, one can show that
$$\mathcal{P}_h(\xi)= \prod_{k=1}^d \frac{e^{-ih\xi_k}-1}{-ih\xi_k}  -\sum_{j=1}^{d} \left\{\frac{e^{-ih\xi_j}-1}{-ih\xi_j} -\frac{4}{h^2\xi_j^2}\sin^{2}\left(\frac{h\xi_j}{2}\right) \right\}\prod_{k\neq j} \frac{e^{-ih\xi_k}-1}{-ih\xi_k}.$$
\end{rem}

\begin{lem}\label{Lem:bound for sin}
If $0<\alpha\le 1$ and $\xi\in[-\tfrac{\pi}{h}, \tfrac{\pi}{h}]^d$, then
$$\Big|e^{-it\left(\frac{4}{h^2}\sum_{j=1}^d \sin^2\left(\frac{h\xi_j}{2}\right)\right)^\alpha}-e^{-it|\xi|^{2\alpha}}\Big|\ls |t|h^2|\xi|^{2\alpha+2}.$$
\end{lem}

\begin{proof}
It is obvious that
\begin{align*}
&\Big|e^{-it\left(\frac{4}{h^2}\sum_{j=1}^d \sin^2\left(\frac{h\xi_j}{2}\right)\right)^\alpha}-e^{-it|\xi|^{2\alpha}}\Big|=\Big|e^{-it\left\{\left(\frac{4}{h^2}\sum_{j=1}^d \sin^2\left(\frac{h\xi_j}{2}\right)\right)^\alpha-|\xi|^{2\alpha}\right\}}-1\Big|\\
&\leq |t|\Big|\Big(\tfrac{4}{h^2}\sum_{j=1}^d \sin^2(\tfrac{h\xi_j}{2})\Big)^\alpha-|\xi|^{2\alpha}\Big|=\tfrac{4^\alpha}{h^{2\alpha}}|t|\Big|\Big(\sum_{j=1}^d \sin^2(\tfrac{h\xi_j}{2})\Big)^\alpha - \big|\tfrac{h\xi}{2}\big|^{2\alpha}\Big|.
\end{align*}
Hence, it is enough to estimate $\big(\sum_{j=1}^d \sin^2(\frac{h\xi_j}{2})\big)^\alpha - \big|\frac{h\xi}{2}\big|^{2\alpha}$. Indeed, if $|\xi|\ll \frac{1}{h}$, then by Taylor's theorem,
\begin{align*}
\Big|\Big(\sum_{j=1}^d \sin^2(\tfrac{h\xi_j}{2})\Big)^\alpha - \big|\tfrac{h\xi}{2}\big|^{2\alpha}\Big|&=\Big|\Big(\big|\tfrac{h\xi}{2}\big|^2+O\left(h^4|\xi|^4\right)\Big)^\alpha - \big|\tfrac{h\xi}{2}\big|^{2\alpha}\Big|\\
&=\big|\tfrac{h\xi}{2}\big|^{2\alpha}\Big(\left(1+O\left(h^2|\xi|^2\right)\right)^\alpha - 1\Big)\\
&\lesssim \left(h|\xi|\right)^{2+2\alpha}.
\end{align*}
On the other hand, if $|\xi|\gtrsim \frac{1}{h}$, then by a trivial inequality, $\big|\big(\sum_{j=1}^d \sin^2(\frac{h\xi_j}{2})\big)^\alpha - \big|\frac{h\xi}{2}\big|^{2\alpha}\big|\leq \big(\sum_{j=1}^d \sin^2(\tfrac{h\xi_j}{2})\big)^\alpha +\big|\tfrac{h\xi}{2}\big|^{2\alpha}
\ls 1\lesssim \left(h|\xi|\right)^{2+2\alpha}$.
\end{proof}

\begin{proof}[Proof of Proposition \ref{linear approx}]
We decompose
\begin{align*}
&p_h e^{-it(-\Delta_h)^\alpha} u_{h,0} - e^{-it(-\Delta)^\alpha} u_0\\
&=\left\{p_he^{-it(-\Delta_h)^\alpha}u_{h,0}-e^{-it(-\Delta)^\alpha} p_hu_{h,0}\right\}+e^{-it(-\Delta)^\alpha}(p_hu_{h,0}-u_0)\\
&=\left\{p_he^{-it(-\Delta_h)^\alpha}u_{h,0}-e^{-it(-\Delta)^\alpha} p_hu_{h,0}\right\}_{low}\\
&\quad+\left\{p_he^{-it(-\Delta_h)^\alpha}u_{h,0}-e^{-it(-\Delta)^\alpha} p_hu_{h,0}\right\}_{high}+e^{-it(-\Delta)^\alpha}(p_hu_{h,0}-u_0)\\
&=:I+II+III,
\end{align*}
where $\widehat{f_{low}}(\xi)=\mathbf{1}_{ |\xi|\leq h^{-\frac{1}{1+\alpha}}}(\xi)\hat{f}(\xi)$ and $f_{high}=f-f_{low}$. For $I$, by the Plancherel theorem and Lemma \ref{p_h symbol} and \ref{Lem:bound for sin}, we get the bound, 
\begin{align*}
\|I\|_{L^2(\mathbb{R}^d)}&=\frac{1}{(2\pi)^\frac{d}{2}}\Big\|\Big\{e^{-it\left(\frac{4}{h^2}\sum_{j=1}^d \sin^2\left(\frac{h\xi_j}{2}\right)\right)^\alpha}-e^{-it|\xi|^{2\alpha}}\Big\}\widehat{p_hu_{h,0}}(\xi)\Big\|_{L^2(\{|\xi| \le h^{-\frac{1}{1+\alpha}}\})}  \\
&\lesssim  \left\||t|h^2|\xi|^{2+2\alpha}\widehat{p_hu_{h,0}}(\xi) \right\|_{L^2(\{|\xi| \le h^{-\frac{1}{1+\alpha}}\})} \\
&\le |t|h^\frac{\alpha}{1+\alpha}\|p_hu_{h,0}\|_{H^\alpha} \lesssim |t|h^\frac{\alpha}{1+\alpha}\|u_{h,0}\|_{H_h^\alpha}.
\end{align*}
For $II$, we use Lemma \ref{linear interpolation inequality} to obtain
\begin{align*}
\|II\|_{L^2(\mathbb{R}^d)}&\lesssim h^{\frac{\alpha}{1+\alpha}}\left\{\|p_he^{-it(-\Delta_h)^\alpha}u_{h,0}\|_{H^\alpha(\mathbb{R}^d)}+\|e^{-it(-\Delta)^\alpha} p_hu_{h,0}\|_{H^\alpha(\mathbb{R}^d)}\right\}\\
&\lesssim h^{\frac{\alpha}{1+\alpha}}\left\{\|e^{-it(-\Delta_h)^\alpha}u_{h,0}\|_{H_h^\alpha(h\mathbb{Z}^d)}+\|p_hu_{h,0}\|_{H^\alpha(\mathbb{R}^d)}\right\}\\
&\lesssim h^{\frac{\alpha}{1+\alpha}}\|u_{h,0}\|_{H_h^\alpha(h\mathbb{Z}^d)}.
\end{align*}
Moreover, we have $\|III\|_{L^2(\mathbb{R}^d)}=\|p_hu_{h,0}-u_0\|_{L^2(\mathbb{R}^d)}$. Collecting all, we complete the proof.
\end{proof}

Finally, we show the following almost distributive law.

\begin{prop}[Almost distributive law for linear interpolation]\label{ph distribution}
	If $p>1$ and $0<\alpha\le1$, then
	\begin{equation}\label{ph L2bound}
	\left\| p_h(|u_h|^{p-1}u_h)(x)-(|p_hu_h|^{p-1}p_h u_h)(x)\right\|_{L^2(\R^d)}\le h^\alpha \| u_h\|_{L_h^\infty(h\Z^d)}^{p-1}\|u_h\|_{H_h^\alpha(h\Z^d)}.
	\end{equation}
	\end{prop}

\begin{proof}
By Lemma \ref{H1 bound depending on h}, it suffices to show the proposition with $\alpha=1$. For $x\in x_m+[0,h)^d$, we write 
\begin{align*}
&p_h(|u_h|^{p-1}u_h)(x) - (|p_hu_h(x)|^{p-1}p_h u_h)(x)\\
&=(|u_h|^{p-1}u_h)(x_m)- (|p_hu_h(x)|^{p-1}p_h u_h)(x)\\
&\quad+  \sum_{j=1}^{d}\frac{(|u_h|^{p-1}u_h)(x_m+he_j)-(|u_h|^{p-1}u_h)(x_m)}{h}(x-x_m)_j.
\end{align*}
Hence, by the fundamental theorem of calculus \eqref{FTC application}, we get
\begin{align*}
&\left|p_h(|u_h|^{p-1}u_h)(x) - (|p_hu_h|^{p-1}p_h u_h)(x)\right|\\
&\lesssim \left(|u_h(x_m)|^{p-1}+|p_hu_h(x)|^{p-1}\right)|u_h(x_m)-p_hu_h(x)|\\
&\quad+ \sum_{j=1}^d \left(|u_h(x_m+he_j)|^{p-1}+|u_h(x_m)|^{p-1}\right)|u_h(x_m+he_j)-u_h(x_m)|\\
&\leq \left(|u_h(x_m)|^{p-1}+|p_hu_h(x)|^{p-1}\right)\left|(D_h^+u_h)(x_m)\cdot(x-x_m)\right|\\
&\quad+ \sum_{j=1}^d \left(|u_h(x_m+he_j)|^{p-1}+|u_h(x_m)|^{p-1}\right)\left|(D_{h;j}^+u_h)(x_m)(x-x_m)_j\right|\\
&\lesssim h \|u_h\|_{L_h^\infty}^{p-1}\left|(D_h^+u_h)(x_m)\right|.
\end{align*}
Thus, integrating its square over $x_m+[0,h)^d$, it follows that
$$\|p_h(|u_h|^{p-1}u_h)(x) - (|p_hu_h|^{p-1}p_h u_h)(x)\|_{L_x^2(x_m+[0,h)^d)}^2\lesssim h^2\|u_h\|_{L_h^\infty}^{2(p-1)}|(D_h^+u_h)(x_m)|^2h^d.$$
Therefore, summing in $x_m$, we prove the proposition.
\end{proof}

\section{Proof of the Main Theorem}

We are now ready to prove Theorem \ref{main theorem}. We fix $u_0$ in $H^\alpha=H^\alpha(\mathbb{R}^d)$, and let $u(t)\in C(\mathbb{R};H^\alpha)$ be the global solution to NLS \eqref{NLS} with initial data $u_0$. With abuse of notation, for each $h>0$, we denote by $u_h(t)$ the global solution (not the discretization of $u(t)$) to DNLS \eqref{DNLS} with initial data $u_{h,0}$, that is, the discretization of initial data $u_0$ (see \eqref{discretization}). Then, we will straightforwardly compare
$$p_hu_h(t)=p_he^{-it(-\Delta_h)^\alpha}u_{h,0}-i\lambda\int_0^t p_he^{-i(t-s)(-\Delta_h)^\alpha}(|u_h|^{p-1}u_h)(s)ds$$
with
$$u(t)=e^{-it(-\Delta)^\alpha}u_{h,0}-i\lambda\int_0^t e^{-i(t-s)(-\Delta)^\alpha}(|u|^{p-1}u)(s)ds.$$

We write the difference as
\begin{align*}
p_hu_h(t)-u(t)&=\left\{p_he^{-it(-\Delta_h)^\alpha}u_{h,0}-e^{-it(-\Delta)^\alpha}u_0\right\}\\
&\quad-i\lambda\int_0^t \left(p_h e^{-i(t-s)(-\Delta_h)^\alpha}-e^{-i(t-s)(-\Delta)^\alpha}p_h\right)(|u_h|^{p-1}u_h)(s)ds\\
&\quad-i\lambda\int_0^t e^{-i(t-s)(-\Delta)^\alpha}\left(p_h(|u_h|^{p-1}u_h)(s)-|p_hu_h|^{p-1}p_hu_h(s)\right)ds\\
&\quad-i\lambda\int_0^t e^{-i(t-s)(-\Delta)^\alpha}\left(|p_hu_h|^{p-1}p_hu_h(s)-|u|^{p-1}u(s)\right)ds\\
&=:I+II+III+IV.
\end{align*}
Indeed, it follows from Proposition \ref{linear approx}, Lemma \ref{discretization inequality} and Proposition \ref{Prop:convergence rate of phgh} that
$$\begin{aligned}\|I\|_{L^2}&\lesssim h^{\frac{\alpha}{1+\alpha}}|t|\left\{\|u_{h,0}\|_{H_h^\alpha(h\Z^d)}+\|u_0\|_{H^\alpha(\R^d)}\right\}+\|p_hu_{h,0}-u_0\|_{L^2(\mathbb{R}^d)}\\
&\lesssim h^{\frac{\alpha}{1+\alpha}}\langle t\rangle\|u_0\|_{H^\alpha}.\end{aligned}$$
For $II$, we apply Proposition \ref{linear approx} and Lemma \ref{linear interpolation inequality} and \ref{Lem:Nonlinear estimate} in order,
\begin{align*}
\|II\|_{L^2}&\lesssim \int_0^t h^{\frac{\alpha}{1+\alpha}}|t-s|\left\{\left\||u_h|^{p-1}u_h(s)\right\|_{H_h^\alpha}+\left\|p_h(|u_h|^{p-1}u_h)(s)\right\|_{H^\alpha}\right\}ds\\
&\lesssim h^{\frac{\alpha}{1+\alpha}}|t|\int_0^t \left\||u_h|^{p-1}u_h(s)\right\|_{H_h^\alpha}ds\\
&\lesssim h^{\frac{\alpha}{1+\alpha}}|t| |t|^{1-\frac{p-1}{q_*}}\|u_h\|_{L_{s}^{q_*}([0,t]; L_h^\infty)}^{p-1}\|u_h\|_{C_t([0,t]; H_h^\alpha)}.
\end{align*}
Thus, by the uniform $L_h^\infty$ bound (Proposition \ref{L^infty bound}) and the uniform $H_h^\alpha$-norm bound (via the conservation laws) on $u_h$, we prove that $\|II\|_{L^2}\lesssim h^{\frac{\alpha}{1+\alpha}}|t|^{2-\frac{p-1}{q_*}}\langle t\rangle^{\frac{p-1}{q_*}}\|u_{h,0}\|_{H_h^\alpha}^p\lesssim h^{\frac{\alpha}{1+\alpha}}\langle t\rangle^2\|u_{0}\|_{H_h^\alpha}^p$. For $III$, it follows from Proposition \ref{ph distribution} that
\begin{align*}
\|III\|_{L^2}&\lesssim \int_0^t \left\||p_hu_h|^{p-1}p_hu_h(s)-p_h(|u_h|^{p-1}u_h)(s)\right\|_{L^2}ds\\
&\lesssim h^{\frac{\alpha}{1+\alpha}}\int_0^t \|u_h(s)\|_{L_h^\infty}^{p-1}\|u_h(s)\|_{H_h^\alpha}ds.
\end{align*}
Then, repeating the argument on $II$, one can show that  $\|III\|_{L^2}\lesssim h^{\frac{\alpha}{1+\alpha}}\langle t\rangle\|u_{0}\|_{H_h^\alpha}^p$. Finally, for $IV$, we apply \eqref{FTC application} to get
\begin{align*}
\|IV\|_{L^2}&\lesssim \int_0^t \left\||p_hu_h|^{p-1}p_hu_h(s)-|u|^{p-1}u(s)\right\|_{L^2}ds\\
&\lesssim \int_0^t \left(\|u_h(s)\|_{L_h^\infty}+\|u(s)\|_{L^\infty}\right)^{p-1}\|p_hu_h(s)-u(s)\|_{L^2}ds.
\end{align*}
Thus, collecting all, we obtain
\begin{align*}
\|u(t)-p_hu_h(t)\|_{L^2}&\lesssim h^{\frac{\alpha}{1+\alpha}}\langle t\rangle^2 \left(1+\|u_{0}\|_{H_h^\alpha}\right)^p\\
&\quad+\int_0^t \left(\|u(s)\|_{L^\infty}+\|u_h(s)\|_{L_h^\infty}\right)^{p-1}\|u(s)-p_hu_h(s)\|_{L^2}ds.
\end{align*}
Finally, applying the Gr\"onwall's inequality, we conclude that 
\begin{align}\label{appendix}
\begin{aligned}
\|u(t)-p_hu_h(t)\|_{L^2}&\lesssim h^{\frac{\alpha}{1+\alpha}}\langle t\rangle^2 \left(1+\|u_{0}\|_{H_h^\alpha}\right)^p \exp\left\{\int_0^t \left(\|u(s)\|_{L^\infty}+\|u_h(s)\|_{L_h^\infty}\right)^{p-1} ds\right\}\\
&\lesssim h^{\frac{\alpha}{1+\alpha}}e^{B|t|}\left(1+\|u_{0}\|_{H_h^\alpha}\right)^p.
\end{aligned}
\end{align}
Indeed, in the last step, the time-averaged $L_h^\infty$ ($L^\infty$, respectively)-bound on $u_h(t)$ ($u(t)$, respectively) are employed (see Proposition \ref{L^infty bound} and Corollary \ref{time-averaged bound for NLS}), and $\langle t\rangle^2$ is absorbed to the exponentially growing term $e^{B|t|}$ since we may take a larger constant $B>0$ for free.

\appendix

\section{Well-posedness theory for (fractional) NLS}

In the appendix, we briefly explain how to get a time averaged $L^\infty$-bound on solutions to NLS, which is similar to a uniform time-averaged $L_h^\infty$-bound on solutions to DNLS \eqref{DNLS}. Indeed, these bounds play a crucial role in the poof of the main theorem (see \eqref{appendix}).

We consider NLS
\begin{equation}\label{NLS2}
\left\{\begin{aligned}
i\partial_t u&=(-\Delta)^\alpha u+\lambda |u|^{p-1}u, \\
u(0)&=u_0\in H^s(\R^d),
\end{aligned}\right.
\end{equation}
where $0<\alpha\leq 1$, $\alpha\neq \frac{1}{2}$ and
$u=u(t,x):\mathbb{R}\times \mathbb{R}^d\to\mathbb{C}$. We recall the following global well-posedness result.

\begin{thm}[Global well-posedness of NLS]\label{WP:NLS}
Suppose that $d$, $\alpha$, $p$ and $\lambda$ obey the hypotheses in Theorem \ref{main theorem}. Then, for any initial data $u_0\in H^\alpha=H^\alpha(\mathbb{R}^d)$, there exists a unique global solution $u(t)\in C_t(\mathbb{R}; H^\alpha)$ to NLS \eqref{NLS2} such that
\begin{equation}\label{NLS a priori bound}
\|u(t)\|_{L_t^q([-T,T]; W^{\alpha-\frac{2(1-\alpha)}{q}, r})}\lesssim \langle T\rangle^{1/q}
\end{equation}
for all admissible pair $(q,r)$ satisfying \eqref{admissible}. 
\end{thm}

\begin{proof}[Sketch of Proof]
One can show local well-posedness by the standard contraction mapping argument. For details, we refer to \cite[Theorem~6.2]{LP} in the case $\alpha=1$, and to \cite[Theorem1.1]{HS} in the case $0<\alpha<1$ and $\alpha\neq\frac{1}{2}$. Indeed, the proof relies on Strichartz estimates 
$$\| e^{-it(-\Delta)^\alpha} f \|_{L_t^q(\mathbb{R};L_x^r(\R^d))} \ls \| |\nabla|^{\frac{2(1-\alpha)}{q}} f \|_{L^2(\R^d)}$$
for admissible $(q,r)$ (see \cite{KT} for $\alpha=1$, and \cite{COX} for $0<\alpha<1$ with $\alpha\neq\frac{1}{2}$). By local well-posedness, the solution $u(t)$ satisfies
$$\|u(t)\|_{L_t^q([-\tau,\tau]; W^{\alpha-\frac{2(1-\alpha)}{q}, r})}\lesssim \|u_0\|_{H^\alpha}$$
on a sufficiently short time interval $[-\tau,\tau]$ depending only on $\|u_0\|_{H^\alpha}$. Moreover, it conserves the mass
$$M(u):=\int_{\mathbb{R}^d}|u(x)|^2\ dx$$
and the energy 
$$E(u):=\int_{\mathbb{R}^d}\frac{1}{2}\left||\nabla|^\alpha u(x)\right|^2+\frac{\lambda}{p+1}|u(x)|^{p+1}\ dx.$$

Because of the choice of $d$, $\alpha$, $p$ and $\lambda$, using the conservation laws and Gagliardo-Nirenberg inequality (as in \eqref{conservation laws and GN applications}), one can show that $\|u(t)\|_{H^\alpha}$ is bounded by a constant depending only on the mass $M(u_0)$ and the energy $E(u_0)$. Thus, the inteval of existence can be extended arbitrarily with the bound \eqref{NLS a priori bound}.
\end{proof}

As an application, we obtain the following desired time-averaged $L^\infty$ bound.
\begin{cor}[Time-averaged $L^\infty$ bound for NLS]\label{time-averaged bound for NLS}
Suppose that $u(t)$ be the global solution to NLS \eqref{NLS2} with initial data $u_0$, constructed in Proposition \ref{WP:NLS}. Then,
$$\|u(t)\|_{L_t^{q_*}([-T,T]; L^\infty)}\lesssim \langle T\rangle^{1/q_*}, $$
where $q_*$ is given by \eqref{q*}. 
\end{cor}

\begin{proof}
One can show the corollary combining \eqref{NLS a priori bound} and the Gagliardo-Nirenbergy inequality. Indeed, in consideration of the formal continuum limit $h\to 0$, one may use the same Lebesgue expoents as in the proof of Corollary \ref{linear L^infty bound}.
\end{proof}

\end{document}